\documentclass[11pt]{amsart}
\usepackage{amsmath, amssymb, amsthm, color}
\usepackage{graphicx}
\usepackage{hyperref}

\newtheorem{theorem}{Theorem}

\newtheorem{lemma}{Lemma}

\newtheorem{definition}{Definition}
\newtheorem{claim}{Claim}

\title{Connectivities for $k$-knitted graphs and for minimal counterexamples to Hadwiger's Conjecture}

\author{Ken-ichi Kawarabayashi}
\address{National Institute of Informatics, 2-1-2 Hitotsubashi, Chiyoda-ku, Tokyo 101-8430, Japan.}
\email{k\_keniti@nii.ac.jp}
\thanks{Research partly supported by Japan Society for the Promotion of Science, Grant-in-Aid for Scientific Research, by C \& C Foundation and by Kayamori Foundation.}

\author{Gexin Yu}
\address{Department of Mathematics, College of William and Mary, Williamsburg, VA 23185.}
\email{gyu@wm.edu}
\thanks{Research supported in part by the NSA grant H98230-12-1-0226.}

\date{\today}

\begin{document}
\maketitle

\begin{abstract}
For a given subset $S\subseteq V(G)$ of a graph $G$, the pair $(G,S)$ is \emph{knitted} if for every partition of $S$ into non-empty subsets $S_1, S_2, \ldots, S_t$, there exist pairwise disjoint connected subgraphs $C_1, C_2, \ldots, C_t$ in $G$ such that $S_i\subseteq V(C_i)$ for all $1 \le i \le t$. A graph $G$ is \emph{$\ell$-knitted} if $(G,S)$ is knitted for every subset $S\subseteq V(G)$ of size $\ell$. In this paper, we prove that every $8\ell$-connected graph is $\ell$-knitted. We subsequently apply this result to Hadwiger's Conjecture, which states that every $k$-chromatic graph contains a $K_k$-minor. Specifically, we demonstrate that the vertex connectivity of any minimal counterexample to Hadwiger's Conjecture is at least $\lceil k/8 \rceil$, improving upon the previous lower bound of $\lceil 2k/27 \rceil$ established by Kawarabayashi (2007).  Our proof corrects a gap in the argument of Kawarabayashi-Yu~(2013) and establishes the claim stated without proof in Liu--Rolek--Yu~(2019).
\end{abstract}

\section{Introduction}

One of the most profound problems in graph theory is Hadwiger's Conjecture, which posits that \emph{every $k$-chromatic graph contains a $K_k$-minor}. A graph $H$ is a minor of a graph $G$ if $H$ can be obtained from a subgraph of $G$ by contracting edges.

It is well known that Hadwiger's Conjecture holds for $k\le 6$. In 1937, Wagner~\cite{W37} proved that the case $k=5$ is equivalent to the Four Color Theorem. Nearly 60 years later, Robertson, Seymour, and Thomas~\cite{RST93} proved that the case $k=6$ is also equivalent to the Four Color Theorem. In their proof, minimal counterexamples---often referred to as ``contraction-critical non-complete graphs''---play a foundational role. 

Kawarabayashi and Toft~\cite{KT05} showed that any $7$-chromatic graph contains either a $K_7$-minor or a $K_{4,4}$-minor, a result for which the connectivity properties of minimal counterexamples are once again pivotal. Extensive research has been dedicated to the connectivity properties of contraction-critical graphs. Dirac~\cite{D52a} proved that every $k$-contraction-critical graph is $5$-connected for $k\ge 5$. Mader~\cite{M68} extended this deep structural result, establishing that every $k$-contraction-critical graph is $7$-connected for $k\ge 7$, and every $6$-contraction-critical graph is $6$-connected. Toft~\cite{T72} proved that $k$-contraction-critical graphs are $k$-edge-connected. Kawarabayashi~\cite{K07} established the first general linear bound on the vertex connectivity of minimal counterexamples to Hadwiger's Conjecture:

\begin{theorem}[Kawarabayashi \cite{K07}]\label{K-thm}
For all positive integers $k$, every minimal (with respect to the minor relation) $k$-chromatic counterexample to Hadwiger's Conjecture is $\lceil \frac{2k}{27}\rceil$-connected.
\end{theorem}

The primary tool used in the proof of Theorem~\ref{K-thm} was the concept of $k$-linked graphs. A graph $G$ is \emph{$k$-linked} if for every $2k$ distinct vertices $u_1, v_1, u_2, v_2, \ldots, u_k,v_k$, there exist $k$ disjoint paths $P_1, P_2, \ldots, P_k$ such that $P_i$ connects $u_i$ and $v_i$. The study of $k$-linked graphs has a rich history and remains central to structural graph theory.

In this paper, we improve the bound in Theorem~\ref{K-thm} by investigating a related notion called a ``knitted graph,'' introduced by Bollob\'as and Thomason~\cite{BT96}. For $1\le m\le k\le |V(G)|$, a graph is \emph{$(k,m)$-knit} if, whenever $S$ is a set of $k$ vertices of $G$ and $S_1, \ldots, S_t$ is a partition of $S$ into $t\ge m$ non-empty parts, $G$ contains vertex-disjoint connected subgraphs $C_1, \ldots, C_t$ such that $S_i\subseteq V(C_i)$ for $1\le i\le t$. Clearly, a $(2k, k)$-knit graph is $k$-linked. In \cite{BT96}, Bollob\'as and Thomason proved that if a $k$-connected graph $G$ contains a minor $H$ with minimum degree at least $0.5(|V(H)|+\lfloor 5k/2\rfloor-2-m)$, then $G$ is $(k, m)$-knit. They used this to show that $22k$-connected graphs are $k$-linked, providing the first linear upper bound on the connectivity required for a graph to be $k$-linked.

We consider a natural generalization of the $(k,m)$-knit property. 

\begin{definition}
For a subset $S\subseteq V(G)$, the pair \emph{$(G, S)$ is knitted} if for every partition of $S$ into non-empty subsets $S_1, S_2, \ldots, S_t$, there exist pairwise disjoint connected subgraphs $C_1, C_2, \ldots, C_t$ in $G$ such that $S_i\subseteq V(C_i)$. A graph $G$ is \emph{$\ell$-knitted} if $(G,S)$ is knitted for all $S\subseteq V(G)$ with $|S|=\ell$.
\end{definition}

It is clear that an $\ell$-knitted graph is $(\ell, m)$-knit for all $m\le \ell$. In this paper, we establish a connectivity threshold for a graph to be $\ell$-knitted.

\begin{definition}\label{separation}
A pair $(A, B)$ of vertex sets is a \emph{separation} of $G$ if $V(G)=A\cup B$ and there are no edges between $A\setminus B$ and $B\setminus A$. The \emph{order} of a separation $(A, B)$ is $|A\cap B|$. If $S\subseteq A$, we say that $(A,B)$ is a separation of $(G,S)$. 
\end{definition}

Our main structural result is the following:

\begin{theorem}\label{knitted-theorem}
Let $k$ and $\ell$ be positive integers, and let $S\subseteq V(G)$ with $|S|=\ell<k/8$. If there is no separation of $(G,S)$ of order less than $\ell$, and every vertex in $G-S$ has degree at least $k-1$, then $(G,S)$ is knitted.
\end{theorem}

Note that in~\cite{KY13}, we claimed a proof of the above result for $|S|=\ell<k/9$. However, we later discovered a gap in the argument, arising from an overly strong conclusion asserted in Theorem~\ref{mass-min}. Subsequently, by replacing the sufficient degree condition used in~\cite{KY13} with the stronger condition established in~\cite{LRY19}, we claimed, without proof, that the above result also holds for $|S|=\ell<k/8$ in ~\cite{LRY19}.

Theorem~\ref{dense-graphs} on edge-density (presented in Section~\ref{thm-4}) is logically stronger than Theorem~\ref{knitted-theorem} and will serve as our primary tool. We are now ready to state and prove our main result on the connectivity of minimal counterexamples to Hadwiger's Conjecture.

\begin{theorem}
For all positive integers $k$, every $k$-chromatic minimal (with respect to the minor relation) counterexample to Hadwiger's Conjecture is $\lceil \frac{k}{8}\rceil$-connected.
\end{theorem}

\begin{proof}
Assume, for the sake of contradiction, that the statement fails. Let $G$ be a minimal $k$-chromatic graph that contains no $K_k$-minor and has vertex connectivity strictly less than $k/8$. Let $S$ be a minimum vertex cutset of $G$. Thus, $|S|=\ell < k/8$. Let $A_1$ be the vertex set of a connected component of $G-S$, and let $A_2 = V(G) \setminus (S \cup A_1)$. 

By the minimality of $G$, both induced subgraphs $G[A_1\cup S]$ and $G[A_2\cup S]$ are proper subgraphs, and therefore have chromatic number at most $k-1$. 

We partition $S$ into independent sets greedily: let $S_1$ be a maximal independent set in $G[S]$, and for $i\ge 2$, let $S_i$ be a maximal independent set in the induced subgraph $G[S\setminus \bigcup_{j=1}^{i-1} S_j]$. Let $S_1, S_2, \ldots, S_t$ be the resulting partition. 

Observe the structural implication of this greedy maximal choice: for any $i > j$, every vertex in $S_i$ must have at least one neighbor in $S_j$, because adding any vertex of $S_i$ to $S_j$ would violate the independence of $S_j$. Consequently, if we find pairwise disjoint connected subgraphs $C_1, \ldots, C_t$ in $G$ such that $S_i \subseteq V(C_i)$, contracting each $C_i$ into a single vertex will form a clique $K_t$ on these contracted vertices.

Since $G$ is a minimal $k$-chromatic graph, its minimum degree is at least $k-1$. Thus, every vertex in $A_p$ (for $p \in \{1,2\}$) has at least $k-1$ neighbors in $A_p\cup S$. Note also that any separation of $(A_1\cup S, S)$ or $(A_2\cup S, S)$ is a valid separation of $(G,S)$. Since $S$ is a minimum cutset of $G$, neither $(A_1\cup S, S)$ nor $(A_2\cup S,S)$ can have a separation of order less than $\ell$.

Therefore, the pairs $(A_1\cup S, S)$ and $(A_2\cup S, S)$ both satisfy the conditions of Theorem~\ref{knitted-theorem}. By Theorem~\ref{knitted-theorem}, both pairs are knitted. 

This guarantees the existence of disjoint connected subgraphs $C_i \subseteq A_1\cup S$ such that $S_i\subseteq V(C_i)$, and similarly, disjoint connected subgraphs $D_i \subseteq A_2\cup S$ such that $S_i\subseteq V(D_i)$. We contract the components $C_i$ within $A_1\cup S$ into single vertices; by our earlier observation, these contracted vertices form a complete graph on the sets $S_1, S_2, \dots, S_t$. Let $G_1$ be the graph obtained from this contraction combined with the uncontracted vertices of $A_2$. Similarly, we can contract the components $D_i$ within $A_2\cup S$ to form a complete graph on the same partition sets, yielding a graph $G_2$ combined with $A_1$.

Both $G_1$ and $G_2$ are proper minors of $G$. By the minimality of $G$, we have $\chi(G_1) \le k-1$ and $\chi(G_2) \le k-1$. Crucially, the contracted vertices corresponding to $S_1, \dots, S_t$ form a clique in both $G_1$ and $G_2$. Therefore, any $(k-1)$-coloring of $G_1$ and $G_2$ must assign all $t$ contracted vertices strictly distinct colors. Since the cliques in $G_1$ and $G_2$ represent the exact same partition of $S$, we can naturally permute the colors of $G_2$ so that they perfectly match the colors assigned to the clique in $G_1$. Fusing these colorings yields a valid proper $(k-1)$-coloring for the entire graph $G$, contradicting the assumption that $G$ is $k$-chromatic. This completes the proof.
\end{proof}

The remainder of the paper is dedicated to proving Theorem~\ref{knitted-theorem}. We achieve this in two steps: first (Section~\ref{thm-4}), we establish that the graphs under study are either knitted or contain a highly dense subgraph; second (Section~\ref{thm-2}), we demonstrate that such dense subgraphs inherently contain a knitted subgraph. This approach structurally mirrors the techniques utilized by Thomas and Wollan~\cite{TW05}.

\section{A theorem on forests with a boundary set}

In this section, we establish the following results concerning the interaction between a forest with a prescribed boundary set and a vertex outside the forest. Although the result is developed for use in our proof, it may also be of independent interest.

\begin{theorem}[General Reduction]\label{thm:unified}
Let $W$ be a graph, and let $S' \subseteq V(W)$ be a boundary set with $|S'| \ge 2$. Let $F$ be a forest in $W$ with $c \ge 1$ connected components such that $S' \subseteq V(F)$ and all leaves of $F$ belong to $S'$. Let $u \in V(W) \setminus V(F)$. Suppose that $d(u, F) \ge |S'| + k$, where $k = 2$ if $c = 1$, and $k = 1$ if $c \ge 2$. Then the induced subgraph $W[V(F) \cup \{u\}]$ contains a forest $F_0$ with $u \in V(F_0)$ such that $S' \subseteq V(F_0)$, all leaves of $F_0$ in $S'$, and at least one of the following holds:
\begin{enumerate}
    \item $|V(F_0)| < |V(F)|$ \quad \text{(Strict Order Reduction)}
    \item $|V(F_0)| = |V(F)|$ and $c(F_0) < c(F)$ \quad \text{(Strict Component Reduction)}
\end{enumerate}
\end{theorem}

\begin{proof}
Partition $S'$ across the components $T_1 \dots T_c$ of $F$ as $S'_i = S' \cap V(T_i)$. Since all leaves of $F$ are in $S'$, each $T_i$ is minimally spanned by $S'_i$. 
Let $N = N_W(u) \cap V(F)$. We have $|N| \ge |S'| + k$. Let $I$ be the index set of components containing at least one vertex of $N$. 

For each $i \in I$ and $s \in S'_i$, let $P_s$ be a shortest path in $T_i$ from $s$ to $N$. Let $y_s \in N$ be its terminus. Define $F_{\text{ess}}^{(i)} = \bigcup_{s \in S'_i} P_s$, and let $Y_i = \{y_s \mid s \in S'_i\}$. Let $Y = \bigcup Y_i$. Clearly, $|Y| \le |S'|$.

We construct an \textbf{intermediate graph} $F'_0$:
\[
F'_0 = \left( \bigcup_{i \in I} F_{\text{ess}}^{(i)} \right) \cup \{u\} \cup \{uy \mid y \in Y\} \cup \left( \bigcup_{j \notin I} T_j \right)
\]
While $F'_0$ connects the active components through $u$ and spans $S'$, it might contain cycles if paths in $F_{\text{ess}}^{(i)}$ coalesce and then diverge to multiple targets in $Y_i$. We resolve this by taking $F_0$ to be a \textbf{minimal spanning sub-forest} of $F'_0$ with respect to $S'$. By minimality, $F_0$ is cycle-free and all leaves of $F_0$ are in $S'$. 

We must guarantee $u \in V(F_0)$. If $u$ were removed during minimization, $F_0$ would be entirely contained within $F \setminus N$. This implies $\bigcup F_{\text{ess}}^{(i)}$ inherently spans $S'$, which forces $F_{\text{ess}}^{(i)} = T_i$ and $N = Y$. However, our threshold guarantees $|N| \ge |S'| + 1$, while $|Y| \le |S'|$, a contradiction. Thus, $u$ survives minimization.

Because $F_0 \subseteq F'_0$, any neighbor of $u$ in $N \setminus Y$ is completely excluded from $F_0$.

If $c = 1$, $k=2$. Thus $|N \setminus Y| \ge |S'|+2 - |S'| = 2$. The order strictly decreases: $|V(F_0)| \le |V(F)| - 2 + 1 < |V(F)|$, satisfying the requirement of our theorem. 

If $c \ge 2$, $k=1$. If $|I|=1$, then $d(u, T_i)\ge |S'|+1\ge |S_i'|+2$ for some sole component $T_i$. It follows that $W[T_i\cup \{u\}]$ has a smaller tree with desired properties. So $|I| \ge 2$.  Then the components merge into one component.  Thus $c(F_0) < c(F)$. Furthermore, $|V(F_0)| \le |V(F)| - |N \setminus Y| + 1 \le |V(F)|$.
\end{proof}

We now classify the exact tight structure when $F$ is a single tree ($c=1$), the degree is lowered to exactly $d(u, F) = |S'| + 1$, and \textit{no} reduction is possible. We will not use this result in our proof, but again, it may be of independent interest. 

\begin{theorem}[Extremal Structure]\label{thm:extremal}
Let $F$ be a tree bounding $S'$ such that $\text{Leaves}(F) \subseteq S'$. Let $u \notin V(F)$ such that $d(u, F) = |S'| + 1$. Suppose $W[V(F) \cup \{u\}]$ contains NO strictly smaller forest $F_0$ spanning $S'$ with $u \in V(F_0)$ and $\text{Leaves}(F_0) \subseteq S'$. 
Then $F$ is a subdivided star with a unique central vertex $z$ of degree $|S'|$, and $N_W(u) \cap V(F)$ consists exactly of $z$ and its $|S'|$ immediate neighbors.
\end{theorem}

\begin{proof}
Let $N = N_W(u) \cap V(F)$ with $|N| = |S'| + 1$. For any valid assignment of each $s \in S'$ to a closest neighbor $y_s \in N$, we can construct the intermediate graph $F'_0 = \bigcup P_s \cup \{u\} \cup \{uy_s\}$ and extract its minimal spanning sub-forest $F_0$.

As established, $|V(F_0)| \le |V(F)| - |N \setminus Y| + 1$. By the hypothesis that no strictly smaller forest exists, we must have $|V(F_0)| \ge |V(F)|$ for \textit{any} valid choice of targets. This forces $|V(F_0)| = |V(F'_0)| = |V(F)|$, and $|N \setminus Y| = 1$. Let $N \setminus Y = \{z\}$. Since $|N| = |S'|+1$, we must have $|Y| = |S'|$, meaning each boundary vertex routes to a distinct unique closest neighbor.

Since $|V(F'_0)| = |V(F)|$, the essential paths must cover exactly $V(F) \setminus \{z\}$. Removing $z$ from $F$ leaves a forest of $d_F(z)$ components. Because the essential paths cover these components and terminate at exactly $|S'|$ distinct targets in $Y$, there are at most $|S'|$ paths, bounding $d_F(z) \le |S'|$.

Consider a component $C$ of $F \setminus \{z\}$. Because $F$ is a tree, $C$ connects to $z$ via a single unique vertex $v_C$. For any $s \in S' \cap V(C)$, the path in $F$ from $s$ to $z$ goes through $v_C$. Since $y_s$ is the closest vertex in $N$ to $s$, we have $d(s, y_s) \le d(s, z)$, which expands to $d(s, v_C) + d(v_C, y_s) \le d(s, v_C) + 1$, forcing $d(v_C, y_s) \le 1$. 

Suppose for contradiction $y_s \neq v_C$. Then $d(v_C, y_s) = 1$, making $y_s$ a branch off $v_C$. In this layout, $y_s$ and $z$ are equidistant from $s$, causing a tie for the closest neighbor in $N$.
If we alternatively selected $z$ as the target for $s$, the new path $P'_s$ would terminate at $z$ instead of $y_s$. Since $y_s$ is a branch off $v_C$, the path $P'_s = s \dots v_C$ strictly excludes $y_s$. This alternative construction removes $y_s$ from the essential forest. The resulting new forest $F^*_0$ would bypass both $y_s$ and whatever unique vertex represents $N \setminus Y^*$, yielding $|V(F^*_0)| \le |V(F)| - 2 + 1 < |V(F)|$.
This directly violates our hypothesis that no strictly smaller forest exists.

Therefore, we must have $y_s = v_C$ universally to prevent tie-breaking reductions. This dictates that $Y$ is precisely the set of immediate neighbors of $z$. Since $Y$ covers all $|S'|$ distinct paths, $d_F(z) = |S'|$. Finally, since the components $C$ are exactly the paths $P_s$ extending from $s$ to $v_C$, $F$ is structurally identical to a subdivided star centered at $z$.
\end{proof}

\section{Dense subgraphs in massed graphs}\label{thm-4}

We introduce a metric to govern our induction. For a set $X\subseteq V(G)$, let $\rho(X)$ denote the number of edges with at least one endpoint in $X$.

\begin{definition}\label{rigid-separation}
A separation $(A, B)$ of $(G,S)$ is \emph{rigid} if the pair $(G[B], A\cap B)$ is knitted.
\end{definition}

\begin{definition}\label{mass}
Let $G$ be a graph, $S\subseteq V(G)$, and $\alpha>1$ be a real number. The pair $(G, S)$ is \emph{$\alpha \ell$-massed} if:
\begin{enumerate}
\item[(i)] $\rho(V(G)-S) \ge  \lfloor\alpha \ell |V(G)-S|\rfloor$, and
\item[(ii)] every separation $(A,B)$ of $(G,S)$ of order at most $|S|-1$ satisfies $\rho(B\setminus A) \le \alpha \ell |B\setminus A|$.
\end{enumerate}
\end{definition}

 From now on, we will try to show that $(G,S)$ is $(s_1, \ldots, s_t)$-knitted, where $\sum_{i=1}^t s_i=\ell$. Since we do not pose any restriction on the partition, this is equivalent to show that $(G,S)$ is knitted. 

\begin{definition}\label{minimal}
The pair $(G,S)$ is \emph{$(\alpha, \ell)$-minimal} if:
\begin{enumerate}
\item $(G,S)$ is $\alpha \ell$-massed, 
\item $|S|\le \ell$ and $(G,S)$ is not $(s_1, \ldots, s_t)$-knitted,
\item subject to the above, $|V(G)|$ is minimized,
\item subject to the above, $\rho(V(G)-S)$ is minimized, and
\item subject to the above, the number of edges of $G$ with both endpoints in $S$ is maximized.
\end{enumerate}
\end{definition}

We state the following theorem, which strengthens the core density arguments.

\begin{theorem}\label{mass-min}
Let $\ell\ge 1$ be an integer and $\alpha\ge 2$ be a real number. If a pair $(G,S)$ is $(\alpha, \ell)$-minimal, then $G$ has no rigid separation of order at most $|S|$. Furthermore, there is a partition of $S=S_1\cup S_2\ldots\cup S_t$ with $s_i=|S_i|$, and a subgraph $H=G[N[v]]$ in $G$ with $|V(H)|\le \lceil 2\alpha \ell\rceil$ such that  
$$d_H(u) \ge 
\begin{cases}
1+\lfloor\alpha \ell\rfloor-s_i, & \mbox{if $u \in S_i$;}\\
1+\lfloor\alpha \ell\rfloor, & \mbox{
if $u \in V(H)-S$}.
\end{cases}
$$
\end{theorem}

\begin{proof}
Since $(G,S)$ is $(\alpha, \ell)$-minimal, $(G,S)$ is not $(s_1, \ldots, s_t)$-knitted. In particular, for a partition $S_1,\ldots, S_t$ of $S$ with $s_i=|S_i|$, we cannot find disjoint connected subgraphs respectively containing $S_i$'s.  By Definition~\ref{minimal} (5), we may assume that vertices between $S_i$ and $S_j$ are all adjacent, since adding edges between $S_i$ and $S_j$ maintain $(G,S)$ to be $(\alpha, \ell)$-minimal.

We structure the proof across three fundamental claims.

\begin{claim}\label{rigid-lemma}
$G$ has no rigid separation of order at most $|S|$.
\end{claim}

\begin{proof} 
Suppose such a rigid separation $(A,B)$ exists; select one minimizing $|A|$.  

First, assume $|A\cap B| < |S|$. Let $G^*[A]$ be the graph obtained by adding all missing edges to make $A\cap B$ a clique. If $(G^*[A], S)$ is knitted, we can effortlessly transform its knits into a valid knit for $(G,S)$ by leveraging the rigidity of $(A,B)$. Since $(G,S)$ is not knitted, $(G^*[A], S)$ is not knitted.
Because $G$ is $\alpha \ell$-massed, $\rho(B\setminus A) \le \alpha \ell|B\setminus A|$. Consequently, $\rho(A\setminus S) > \alpha \ell|A\setminus S|$. This means condition (i) of being massed is satisfied, forcing condition (ii) to fail in $G^*[A]$.
Let $(A', B')$ be a separation of $G^*[A]$ containing $S \subseteq A'$ that violates (ii), minimizing $B'$. If $A\cap B \subseteq A'$, then $(A'\cup B, B')$ acts as a separation in $G$ violating (ii), which is impossible. Thus, since $A\cap B$ is a clique, $A\cap B \subseteq B'$. 
By the minimality of $B'$, the subgraph $(G^*[B'], A'\cap B')$ is itself massed and thus knitted. Consequently, $(G^*[B\cup B'], A'\cap B')$ is knitted, making $A'\cap B'$ a rigid separation of $(G,S)$ strictly smaller than $A$, contradicting our minimal choice of $A$.

Next, assume $|A\cap B| = |S|$. If there are $|S|$ disjoint paths from $S$ to $A\cap B$, the rigidity of $(A,B)$ instantly makes $(G,S)$ knitted. Otherwise, there is a separation $(A'', B'')$ of $(G[A], S)$ of order less than $|S|$ separating $S$ from $A\cap B$. Choosing the minimal $|A''\cap B''|$, we generate disjoint paths from $A''\cap B''$ to $A\cap B$. The rigidity flows back, establishing $(A'', B\cup B'')$ as a rigid separation of $(G,S)$ with $|A''| < |A|$, again contradicting minimality.
\end{proof}

\begin{claim}\label{lemma-common}
Let $uv$ be an edge with $v \notin S$. Then $d(v) \ge 1+\lfloor\alpha \ell\rfloor$. Moreover, if $u \notin S$, then $u$ and $v$ have at least $\lfloor\alpha \ell\rfloor$ common neighbors. If $u \in S$, then $u$ and $v$ have at least $\lfloor\alpha \ell\rfloor - |N(v)\cap S-N[u]|$ common neighbors.
\end{claim}

\begin{proof}
Let $G'=G/uv$ be the graph resulting from contracting the edge $uv$. If $(G', S)$ is knitted, $(G, S)$ inherits the knit. Thus $(G', S)$ must violate (i) or (ii).

Suppose $(G', S)$ violates (i). Then $\rho(V(G')-S) \le \lfloor\alpha \ell |V(G')-S|\rfloor - 1$. Since $G$ satisfies (i), $\rho(V(G)-S)\ge \lfloor\alpha \ell |V(G)-S|\rfloor$. The contraction reduces the vertex count by exactly one.
\[ \rho(V(G)-S) - \rho(V(G')-S)\ge 1+\lfloor\alpha\ell |V(G)-S|\rfloor - \lfloor\alpha\ell|V(G')-S|\rfloor 
\ge 1+\lfloor\alpha\ell\rfloor. \]
The edges lost in $\rho(V(G)-S)$ during contraction are exactly the edge $uv$, one duplicate edge for every common neighbor of $u$ and $v$ not in $S$, and (if $u \in S$) any edge from $v$ to $S$ that becomes redundant. All such edges are incident to $v$, establishing $d(v) \ge \rho(V(G)-S) - \rho(V(G')-S)\ge 1+ \lfloor\alpha\ell\rfloor$. 
If $u \notin S$, the exact number of lost edges is $1 + |N(u) \cap N(v)|$, yielding $|N(u) \cap N(v)|\ge\lfloor \alpha \ell\rfloor$. If $u \in S$, the lost edges number $1 + |N(u) \cap N(v)| - |N(v) \cap S - N[u]|$, preserving the inequality.

Alternatively, assume $(G',S)$ violates (ii). Let $(A',B')$ be the violating separation with minimal $B'$. By minimality, $(G'[B'], A'\cap B')$ is knitted, meaning $(A', B')$ acts as a rigid separation in $G'$. In the pre-contraction graph $G$, this induces a separation $(A,B)$. If $\{u,v\} \not\subseteq A\cap B$, then $(A,B)$ operates as a rigid separation of $(G,S)$ of order at most $|S|-1$, decisively violating Claim~\ref{rigid-lemma}. If $u,v \in A\cap B$, the minimality of $B'$ ensures $(G[B], A\cap B)$ is $\alpha\ell$-massed and therefore knitted, creating a rigid separation of size $\le |S|$, violating Claim~\ref{rigid-lemma} once more.
\end{proof}

\begin{claim}
Let $\delta'$ be the minimum degree in $G$ among vertices in $V(G)-S$. Then $\alpha \ell < \delta' < 2\alpha \ell$. 
\end{claim}

\begin{proof}
Claim~\ref{lemma-common} guarantees $\delta' > \alpha \ell$. We must prove $\delta' < 2\alpha \ell$. Let $v \in V(G)-S$ be an arbitrary vertex. Since $d(v) > \alpha \ell \ge 2\ell > |S|$, $v$ possesses at least one neighbor $u \in V(G)-S$. Let $e = uv$, and consider $G_1 = G-e$. Because $G$ is minimal, $G_1$ must fail (i) or (ii).

If $G_1$ fails (ii), $(G-e, S)$ contains a separation $(A,B)$ with $|A\cap B|<|S|$. The edge $e$ must cross this separation, forcing $u \in A\setminus B$ and $v \in B\setminus A$. Therefore, their common neighbors are restricted to $A\cap B$. This implies $|N(u)\cap N(v)| \le |A\cap B| < |S| \le \ell < \alpha \ell$, contradicting Claim~\ref{lemma-common}.

Thus, $G_1$ must fail (i), yielding $\rho(V(G_1)-S) \le \lfloor\alpha \ell |V(G)-S|\rfloor - 1$. Since removing $e$ drops the edge count in $V(G)-S$ by precisely 1, we find exactly $\rho(V(G)-S) \le \lfloor\alpha \ell |V(G)-S|\rfloor$.


The sum of the degrees of the vertices in $V(G)-S$ satisfies:
\begin{align*} \sum_{w \in V(G)-S} d_G(w) &= 2|E(V(G)-S)| + |E(V(G)-S, S)| = 2\rho(V(G)-S) - |E(V(G)-S, S)|\\ 
&\le 2\lfloor\alpha \ell |V(G)-S|\rfloor -|E(V(G)-S, S)| < 2\lfloor\alpha \ell |V(G)-S|\rfloor\le 2\alpha \ell |V(G)-S|. \end{align*}
since $|E(V(G)-S, S)|> 0$.
By the Pigeonhole Principle, the minimum degree among these vertices, $\delta'$, must be strictly less than $2\alpha \ell$.
\end{proof}

Finally, select a vertex $v\in V(G)-S$ with degree exactly $\delta'$. Let $H$ be the subgraph induced by $v$ and its neighborhood $N(v)$. The number of vertices in $H$ is $\delta'+1<1+ 2\alpha \ell$. 
By Claim~\ref{lemma-common}, any $u \in V(H)$ with $u \notin S$ has $d_H(u) = 1 + |N(u) \cap N(v)| > \alpha \ell$. For any $u \in S_i$, its degree within $H$ is strictly greater than $\alpha \ell - 1 - |N(v) \cap S - N[u]| + 1 \ge \alpha \ell - |S_i|+1$, since $u$ is adjacent all vertices in $S-S_i$. 
This constructed subgraph $H$ completely fulfills all requirements of the theorem.
\end{proof}

\section{Dense graphs contain  knitted subgraphs}\label{thm-2}

In this section, we show the dense graph $H$ from Section~\ref{thm-4} contains a knitted subgraph. This is a critical ingredient for a proof of our main result.

We will make use of the following result:

\begin{theorem}[Liu, Rolek, and Yu~\cite{LRY19}]\label{LRY}
Let $G$ be a graph with $n \ge 2\ell + 3$ vertices, where $\ell \ge 5$. If $\delta(G) \ge \frac{n+\ell}{2} - 1$, then $G$ is $\ell$-knitted.
\end{theorem}

\begin{theorem}\label{subgraph}
Let $\ell \ge 5$ be an integer, and let $\alpha \ge 4$. Let $H$ be a graph with $|V(H)| \le  \lceil 2\alpha \ell \rceil$, let $T \subseteq V(H)$ with $|T| \le \ell$. For any $S \subseteq V(H)$ such that $T \subseteq S$ and $|S| = \ell$, and a partition of $S=\cup_{i=1}^t S_i$ with $s_i=|S_i|$, if $v \in V(H)$, 
\[
d(v)\ge 
\begin{cases}
1+\lfloor\alpha\ell\rfloor,  &  \mbox{if $v \in V(H) - T$};\\
    1+\lfloor\alpha \ell\rfloor-s_i, & \mbox{if $v \in T\cap S_i$}.
\end{cases}
\]
then $H$ contains an $(s_1, \ldots, s_t)$-knitted subgraph.
\end{theorem}

\begin{proof}[Proof of Theorem~\ref{subgraph}]
Assume, for the sake of contradiction, that there is a subset $S \subseteq V(H)$ with $T \subseteq S$ and $|S|=\ell$, and a partition $S=\bigcup_{i=1}^{t} S_i$ such that $H$ does not contain disjoint connected subgraphs corresponding to each $S_i$. We analyze a \emph{partial} $(\ell, t)$-knit $C=\bigcup_{i=1}^t C_i$, which is a subgraph of $H$ wherein $S_i\subseteq V(C_i)$ for all $i$, but the components $C_i$ are not strictly required to be connected.

We define an \emph{optimal} $(\ell, t)$-knit $C=\bigcup_{i=1}^t C_i$ as a partial $(\ell, t)$-knit that satisfies the following conditions in decreasing order of priority:
\begin{enumerate}
\item[(a)] $|V(C)|\le \min_{z\in S} d_H(z)$;
\item[(b)] the total number of connected components of $C$ is minimized;
\item[(c)] subject to (a) and (b), $|V(C)|$ is minimized.
\end{enumerate}

By the optimality of $C$, any component of $C$ containing exactly one vertex of $S$ must be an isolated vertex, and any component containing exactly two vertices of $S$ must be a path connecting them. We may assume without loss of generality that $S_1\subseteq V(C_1)$, but $C_1$ is not connected. Thus, there exist $x, y\in S_1$ belonging to different connected components of $C_1$.

Since  $|V(C)|\le \min_{z\in S} d_H(z)\le d(x)$ and $x\in C$, $d(x, H-C)\ge 1$; in particular, $H-C \neq \emptyset$.


\begin{lemma}\label{delta-star}
$\delta^* \ge (\alpha - 1)\ell-c$, where $c=|\{s_i\ge 2: 1\le i\le t \mbox{ and $C_i$ is connected}\}|$. 
\end{lemma}

\begin{proof}
For every $u\in V(H-C)$, we have $$d(u, H-C) = d(u,H) - d(u,C)\ge 1+\lfloor\alpha\ell\rfloor- d(u,C)\ge \alpha\ell-d(u,C).$$ 

We now bound $d(u,C)$. Let $C_i$ be the forest containing $S_i$ in $C$. Then if $s_i=1$, the component $C_i$ is a single vertex, so $d(u, C_i) = 1 = s_i$; if  $s_i\ge 2$, Theorem~\ref{thm:unified} implies $d(u, C_i)\le s_i+1$, with equality only if $C_i$ is a tree and $s_i\ge 2$. It follows that $$d(u, C) = \sum_{C_i} d(u,C_i) \le |S|+c=\ell+c.$$ The lemma follows.
\end{proof} 

Since $x,y$ are not in the same component, so $c\le \frac{\ell-2}{2}$;  furthermore, for each $S_i\subseteq S$ with $x,y\not\in S_i$, $$c\le \frac{\ell-2-s_i}{2}+1=\frac{\ell-s_i}{2}.$$

\begin{lemma}
The subgraph $H-C$ is connected.
\end{lemma}

\begin{proof}
Suppose that $H-C$ has components $H_1, \ldots, H_p$ where $p\ge 2$. Since $H_i$ does not inherently contain an $\ell$-knit, and $|V(H_i)|\ge \delta^*\ge 2.5\ell$, by Lemma~\ref{delta-star} and $c\le \ell/2$ and $\alpha\ge 4$. By Theorem~\ref{LRY}, $\delta^* < (|V(H_i)|+\ell)/2-1$. It follows that $$|V(H_i)| > 2\delta^*-\ell+2\ge (2\alpha-3)\ell-2c+2\ge (2\alpha-4)\ell+2.$$
Since $p \ge 2$, we sum the sizes:
$|V(H)| \ge |V(C)| + |V(H_1)| + |V(H_2)| > \ell + 2(2\alpha-4)\ell + 4$.
Given $|V(H)| < \lceil 2\alpha \ell\rceil$, this implies $\lceil 2\alpha \ell\rceil > (4\alpha-7)\ell + 4$, which simplifies to $(2\alpha-7)\ell\le -3$. This contradicts our hypothesis that $\alpha \ge 3.5$.
\end{proof}

\begin{lemma}
$|V(C)| \ge \min_{z\in S} d_H(z)-4$. 
\end{lemma}

\begin{proof}
Assume that $|V(C)| \le \min_{z\in S} d_H(z)-5$. 
Let $A=N(x)\cap V(H-C)$ and $B=N(y)\cap V(H-C)$. Let $A'=N(A)\cap V(H-C)-A$ and $B'=N(B)\cap V(H-C)-B$. Finally, let $D=V(H-C)-(A\cup A'\cup B\cup B')$. Since $H-C$ is connected, there is a path from $x$ to $y$ through $A\cup A'\cup D\cup B'\cup B$. If such a path has length at most $6$, we could append this path to $C$ to form a new knit $C'$. This new knit $C'$ would strictly reduce the number of connected components. Furthermore, its size would be bounded by $|V(C')| \le |V(C)| + 5 \le \min_{z\in S} d_H(z)-5+5=\min_{z\in S} d_H(z)$. This directly contradicts the assumption that $C$ is an optimal knit.

Since no such short path exists, take $u\in D-N(A')$; $u$ has no neighbors in $A'\cup A$. Take $v\in A$; the pair $u, y$ share no common neighbors, nor do $v, y$. 
Consequently, we can bound the size of $H$:
\begin{align*} 
|V(H)| &\ge 3+d(y) + d(u, H-C) + d(v, H-C) \ge 3+\delta(H) + 2\delta^* \\
&\ge 3+\alpha\ell-s_i + (2\alpha - 2)\ell -2c= 3+(3\alpha - 2)\ell-s_i-2c\\
&\ge 3+(3\alpha - 2)\ell-s_i-2\cdot (\ell-s_i)/2=3+(3\alpha-3)\ell.
\end{align*}
Since $|V(H)|\le \lceil 2\alpha\ell\rceil\le 2\alpha\ell + 1$, this forces $2\alpha\ell + 1 \ge 3+(3\alpha-3)\ell$, yielding $(\alpha-3)\ell\le -2$. This blatantly contradicts $\alpha \ge 3$. 
\end{proof}

Therefore, $|V(C)| \ge \min_{z\in S} d_H(z)-4$. Then we have:
\[ |V(H-C)| = |V(H)| - |V(C)| \le \lceil 2\alpha \ell\rceil - \min_{z\in S} d_H(z)+4. \]
Since $|V(H-C)|\ge \delta^*+1\ge (\alpha-1)\ell-c+1\ge 2\ell+3$, we try to apply Theorem~\ref{LRY}. Since $H-C$ is not $\ell$-knitted, we have 
\begin{align*}
0&>2\delta^* - \left(|V(H-C)| + \ell - 2 \right)\\
&\ge 2+2\lfloor\alpha\ell\rfloor-2d(u,C) - \big(\lceil 2\alpha \ell\rceil - \min_{z\in S} d_H(z)+4\big)- \ell + 2 \\
&= \min_{z\in S} d_H(z)-2d(u,C)+(2\lfloor\alpha\ell\rfloor-\lceil 2\alpha \ell\rceil-\ell)\\
&\ge (1+\lfloor\alpha\ell\rfloor-s_i)-2d(u,C)-\ell-(2\lfloor\alpha\ell\rfloor-\lceil 2\alpha \ell\rceil)\\
&\ge (1+\lfloor\alpha\ell\rfloor)-s_i-2(\ell+c)-\ell-(2\lfloor\alpha\ell\rfloor-\lceil 2\alpha \ell\rceil)\\
&\ge (1+\lfloor\alpha\ell\rfloor)-s_i-2\ell-2\cdot (\ell-s_i)/2-\ell-(2\lfloor\alpha\ell\rfloor-\lceil 2\alpha \ell\rceil)\\
&\ge (1+\lfloor\alpha\ell\rfloor)-4\ell-(2\lfloor\alpha\ell\rfloor-\lceil 2\alpha \ell\rceil)\\
&\ge (\alpha-4)\ell.
\end{align*}
where $d(u,H-C)=\delta^*$ for $u\in H-C$, $\min_{z\in S} d(z)=d(z)\ge 1+\lfloor\alpha\ell\rfloor-|S_i|$ for $z\in S_i$. We used that $d(u,C)=\sum_j d(u,C_j)\le \ell+c$, and $c\le (\ell-|S_i|)/2$.
A contradiction to $\alpha\ge 4$.
\end{proof}

Combining Theorem~\ref{mass-min} and Theorem~\ref{subgraph}, we deduce the following result, which encompasses Theorem~\ref{knitted-theorem}.

\begin{theorem}\label{dense-graphs}
Let $\ell$ be an integer. Let $G$ be a graph and $S\subseteq V(G)$ be a subset of size $\ell$. If $(G,S)$ is $(4, \ell)$-massed, then $(G,S)$ is knitted. 
\end{theorem}

\begin{proof}
Assume, for the sake of contradiction, that there exists a $(4, \ell)$-massed graph that is not $(s_1, \ldots, s_t)$-knitted. We may select $G$ such that $(G,S)$ is $(4,\ell)$-minimal. By Theorems~\ref{mass-min} and \ref{subgraph}, $G$ contains no rigid separation of order at most $\ell$, but it strictly contains an $(s_1, \ldots, s_t)$-knitted subgraph $K$.  Let $S_1, \ldots, S_t$ with $s_i=|S_i|$ be a partition of $S$ that is not knitted. 

By Menger's Theorem, if there are $\ell$ vertex-disjoint paths from $S$ to $K$, we can uniquely associate each vertex in $S$ to a distinct terminal in $K$ (take the first vertices on the paths that are on $K$). Take the corresponding partition of the terminals in $K$ which can be connected via disjoint connected subgraphs inside $K$. Tracing back through the paths resolves the knit in $G$.

If no such $\ell$ disjoint paths exist, there must be a separation $(A,B)$ separating $S$ from $K$ with $S\subseteq A$ and $K\subseteq B$, such that its order $|A\cap B|$ is strictly less than $\ell$. We choose $(A,B)$ to have the minimum possible order. By Menger's Theorem again, there are $|A\cap B|$ disjoint paths from $A\cap B$ to $K$. The identical routing logic demonstrates that for any partition of $A\cap B$, the parts can be connected via $K$. Thus, the induced subgraph $G[B]$ securely knits $A\cap B$. This perfectly satisfies Definition~\ref{rigid-separation}, meaning $(A,B)$ is a rigid separation of order at most $\ell-1$, squarely contradicting the guarantee of Theorem~\ref{mass-min}.
\end{proof}

\section*{Acknowledgement} 
The authors would like to thank Michael Lafferty for the initial discussions that contributed to this revision.

\end{document}